\documentclass[reqno,twoside,final]{amsart}
\usepackage{amsmath,amstext,amsthm,amsfonts,amssymb,amscd}
\usepackage[numbers]{natbib}

\theoremstyle{definition}

\newtheorem{rem}{Remark}
\newtheorem*{rem*}{Remark}
\newtheorem*{acknow*}{Acknowledgements}
\newtheorem*{examples*}{Examples}

\theoremstyle{plain}

\newtheorem{lemma}{Lemma}

\newtheorem{theorem}{Theorem}
\newtheorem*{theorem*}{Theorem}
\newenvironment{proof-sketch}{\noindent{\bf Sketch of Proof}\hspace*{1em}}{\qed\bigskip}
\newenvironment{proof-idea}{\noindent{\bf Proof Idea}\hspace*{1em}}{\qed\bigskip}
\newenvironment{proof-of-lemma}[1]{\noindent{\bf Proof of Lemma #1}\hspace*{1em}}{\qed\bigskip}
\newenvironment{proof-of-prop}[1]{\noindent{\bf Proof of Proposition #1}\hspace*{1em}}{\qed\bigskip}
\newenvironment{proof-of-thm}[1]{\noindent{\bf Proof of Theorem #1.}\hspace*{1em}}{\qed\bigskip}
\newenvironment{proof-attempt}{\noindent{\bf Proof Attempt}\hspace*{1em}}{\qed\bigskip}

\begin{document}

\title{Algebro-geometric proof of Christoph's Theorem}
\author{Sergey Malev and and Anastasiia Zhilina}

\address{Department of mathematics, Ariel University,
Ariel, Israel} \email{sergeyma@ariel.ac.il}
\address{Higher School of Economics, Moscow, Russian Federation}
\email {tasya020822@mail.ru}
\thanks{This research was supported by the Russian Science Foundation grant No. 22-11-00177, https://rscf.ru/project/22-11-00177}




\maketitle

\subsection*{Abstract}
In this paper an algebraic proof of Christoph's theorem is provided. This theorem from algebraic-geometry is about the existence of a finite automaton for computing coefficient of a series for an algebraic function.



\section{Introduction}
The study of Taylor series is related to the problem of embedding algebraic varieties (see \cite{Ch, BC}), to algorithmic issues and issues of algebra bases (\cite{BRV}).

The formulation of Christoph's theorem includes the algebraic and logical sections of mathematics. It would be expected that it has algebro-geometric proof. 
\begin{theorem}
    Let $P(x) = a_0 + a_1x + a_2x^2 + \ldots$ be a power series with coefficients from the field of finite characteristic $p$. This series is an algebraic series if and only if $P(x)$  has a finite automaton,which given a number $n$ in a number $p$ system and returns $a_{n}$.
\end{theorem}
There will also be shown will be shown the connection of Christoph's theorem and Riemann-Roch theorem. Riemann-Roch theorem is one of the fundamental theorems in algebraic geometry.
\begin{theorem}
    If $A$ is arbitrary field divisor $K$ and $\lambda$ is arbitrary non-zero differential, then $$\displaystyle \ell(A) = n(A) - g + 1 + \displaystyle \ell(A^{-1}D_{\lambda})$$ where  ${\displaystyle \ell (D)}$ is the dimension (over C) of the vector space of meromorphic functions $h$ on a surface such that all coefficients of the divisor (h) + D are non-negative, and $n(D)$ is divisor dimension.
\end{theorem}
For the proof of Christoph’s theorem will be use similar fact, but for fields of finite dimension. Also we will look only on algebraic functions - that is, algebraic extension of the set of rational functions.
\begin{lemma}\label{Lemma1}
    Let $A$ denote a set of algebraic functions with fixed set of poles and suppose each of the poles have degree less than $k$. In this case $A$ have a finite basis.
\end{lemma}
Book  D. Mumford~--\cite{Mam} was used for studying algebraic geometry. Also a lot of facts about the valuation theory related to Lemma(\ref{Lemma1}) are detailed in the book on algebra written by B. L. Van Der Waerden~--\cite{Var}.

\section{Proof of central lemma}

To begin with, we prove the Lemma(\ref{Lemma1}). For this purpose, let us introduce few definitions.

We call a function $\omega$ \emph{rationing} of the field $F$ called function from $F$ to $\mathbb R ^{*}$ if it satisfies the following: 

1) $\omega(a)=\infty$ $\Leftrightarrow$ $a=0$

2) $\omega(ab)=\omega(a) + \omega(b)$

3) $\omega(a + b) \geqslant min(\omega(a), \omega(b)$

Rationings $\omega$ and $\psi$ are equivalent, if when $\omega(a) < 1$, we also have $\psi(a) < 1$.

Place of field $F$ is an equivalence class of valuations of this field. Within the framework of the problem posed, we will only talk about valuations of the field of algebraic functions. That is, about extensions of valuations of the field of rational functions to its extension.

The divisor is a formal finite product of places $p_{i}$ in some integer degrees $d_{i}$. Divisor of function $f$ is $(f) =\prod_{i=1}^{n} p_{i}^{d_{i}}$, with $d_{i}$~-- order $f$ of the place $p_{i}$.

Also, for the proof, we need the theorem on functions without places with non-zero order:

\begin{theorem}\label{theorem3}
    Function $f$ having no poles is a constant.
\end{theorem}

The proof of this theorem is given in the book by D. Mumford\cite{Mam}, in chapter 19. 

The function $f$ is a multiple of the divisor $A$, if all degrees of the divisor $(f) A^{-1}$ are greater than 0. Thus, it suffices to prove, that functions $f$, that are multiples of the divisor $A^{-1}$, have a finite basis, where $A = \prod _{i=1}^{n} p_{i}^{d_i}$ and $d_i < k$ for all $t$.

\begin{proof}[Proof of Lemma 1]
    Suppose there are $f_1, \ldots f_{n+2}$ linearly independent multiples of $A$, where $ n = \sum d_i $. If we decompose all of $f_i$ to the series, we can choose $a_1, \ldots a_{n+2}$ such that function $f = a_1f_1 + \ldots + a_{n + 2}f_{n + 2}$ has only positive non-zero coefficients, because there are $n$ linear conditions for $a_{i}$, and $n + 2$  unknown variables. Moreover, the functions $f$ with this condition form  modulus of dimension not less than 2. But functions without negative degrees in series expansion haven't got poles, hence, from the Theorem(\ref{theorem3}) we have that all $f$ are constants. It contradicts with the fact that the constants are a module of dimension 1.  
\end{proof} 

Thus, algebraic functions with a finite number of bounded poles have a finite basis.

\section{Transformation of series} 

Let us look at the series $P(x) = a_0 + a_1 x + a_2 x \ldots$. We work in characteristic $p$, so if we take $(p - 1)$-th derivative than the new series will look like $P(x)^{(p - 1)} = a_{p-1} + a_{2p-1}x^{p} + a_{3p-1}x^{2p} + \ldots$, because $1\cdot 2 \cdot\ldots\cdot(p-1) = -1$. Note, that in characteristic $p$  $a^p + b^p = (a + b)^p$, so from the expression $P(x)^{(p - 1)}$ one can take the root of the $p$~-th degree term by term. 

\begin{equation}
\sqrt[{p}]{-P(x)^{(p - 1)}} = \sqrt[{p}]{a_{p-1}} + \sqrt[{p}]{a_{2p-1}}x + \sqrt[{p}]{a_{3p-1}}x^2 + \ldots = a_{p-1} + a_{2p-1}x + a_{3p-1}x^2 + \ldots
\end{equation}

The same transformation can be used with first multiplying $P(x)$ by $x^{k}$, where $k<p$. This transformation will be called a \emph{weeding} of degree $k$. This transformation can also be used for algebraic functions. Weeding of degree $k$ for $f$ will be denoted by $f^{[k]}$.

\begin{rem} \label{Remark1}
       
    The operation described above is linear, that is, for two functions $f$ and $g$ it is true that $\alpha f^{[k]}+g{[k]} = (\alpha f+g)^{[k]}$ for all algebraic functions $f$ and $g$ and for every $\alpha$ from the field.
        
\end{rem}
$\\$
Let's also see how the divisor of the function changes when a weeding is performed.
\begin{rem}\label{Remark2}
    If we apply all possible compositions of weedings of various degrees, then we get a set of algebraic functions with a finite number of poles of a bounded degree.
\end{rem}

\begin{proof}[Proof of Remark 2]

    Initially, an algebraic function has a finite number of poles. If take $p-1$ derivatives,  power of the finite poles will increase by at most $p$ times. Multiplication by $x^k$ will only increases the degree of a pole at infinity by at most $k<p$. And when taking the root of the $p$-th degree, the number of poles will decrease by a factor of $p$, thus, after weeding, the degree of all poles is bounded.
\end{proof} 

\section{Proof of Chistoph's theorem} 

\begin{theorem} \label{3}

	Let $P(x) = a_0 + a_1x + a_2x^2 + \ldots$ be a power series with coefficients from the field of finite characteristic $p$. This series is an algebraic series if and only if $P(x)$  has a finite automaton,which given a number $n$ in a number $p$ system and returns $a_{n}$.

\end{theorem} 

\begin{proof} 
    
     Consider a number $n$ whose $p$-ary system is $k_0 + p  k_1 \ldots + p^{i}{k_{i}}$. We do a weeding of degree $p - k_{0} - 1$ of $P(x)$. Note that since it follows from the $p$-ary system of the number $n$ that ${n-k_{i}}$ is divisible by $p$, under the chosen weeding the coefficient of $x^{n}$ will turn into $p$-th power of the coefficient of $x^{\frac{n + p - k_{0} - 1 - (p - 1)}{p}}$.
    
    Denote $\frac{n - k_{0}}{p}$ as $n_1$, so the $p$-exact representation of $n_1$ is $k_1 + p k_2 \ldots + p^{i - 1}{k_{i}}$. 
    
    Successively applying weedings of degree $p - k_{0} - 1$, $p - k_{1} - 1$, $\ldots$, $p - k_{i} - 1$, we obtain a new series, the free coefficient of which will be equal to $a_n$.
    
    Let's use Remark($\ref{Remark2}$). If we apply all possible compositions of weedings of various degrees, then we get a set of algebraic functions with bounded degrees of poles, in other words, all the resulting functions will be multiples for some divisor. By virtue of Lemma($\ref{Lemma1}$), we can state that then all possible functions obtained by applying weedings are finitely based. Let choose a basis of this space $z_1$, $z_2$, $\ldots$, $z_m$, these functions correspond to some series $z_i = b_{i_0}+b_{i_1}x+b_{i_2}x^2+\ldots $.
    
    Finally we need to build a finite state machine:
    
    Consider a finite automaton for the series $P(x)$ that takes as input a sequence of numbers from 0 to $p-1$~-- the coefficients of the $p$-ary representation of some number $n$ and returns an element of the field $F$ equal to $a_n$ . After each step, the automaton will store a sequence of $m$ numbers $\alpha_1, \ldots, \alpha_m$ initially equal to the expansion coefficients of $P(x)$ in terms of the basis $z_1$, $z_2$, $\ldots$, $z_m$.
    
    When a new element $k$ is read, the stored list will be updated by the sum 
    \begin{equation}
    \alpha_1(a_{1_1}, a_{1_2}, \ldots,a_{1_m})+\alpha_2(a_{2_1}, a_{2_2}, \ldots,a_{2_m})+\ldots+\alpha_m(a_{m_1}, a_{m_2}, \ldots,a_{k_m})
    \end{equation}
    where $(a_{i_1}, a_{i_2}, \ldots,a_{i_m})$ is an expansion of $z_i^{[p - k - 1]}$ in the chosen basis. This will be equivalent to weeding of the degree $k$ of the entire series, due to Remark($\ref{Remark1}$).
    
    When the input runs out, this automaton will automatically return the value $a_n = \alpha_1b_{1_0} + \ldots + \alpha_m b_{m_0}$.
    
    Thus, we have built a finite automaton which, given the $p$-ary representation of the number $n$, returns the coefficient $P(x)$ for $x^n$.
   	
\end{proof}

\end{document}